\documentclass[12pt]{amsart}
\usepackage{xcolor}
\usepackage{amsmath}
\usepackage{amssymb}
\usepackage{amsfonts} 
\usepackage{mathrsfs}
\usepackage{mathtools}

  \newcommand{\Edb}{\mbox{$\mathbb{E}$}}

   \newcommand{\Ndb}{\mbox{$\mathbb{N}$}}
   
   \newcommand{\Pdb}{\mbox{$\mathbb{P}$}}

   \newcommand{\Zdb}{\mbox{$\mathbb{Z}$}}

   \newcommand{\E}{\mbox{${\mathcal E}$}}

   \newcommand{\N}{\mbox{${\mathcal N}$}}
   
   \renewcommand{\P}{\mbox{${\mathcal P}$}}

\newcommand{\norm}[1]{\Vert#1\Vert}
\newcommand{\bignorm}[1]{\bigl\Vert#1\bigr\Vert}
\newcommand{\Bignorm}[1]{\Bigl\Vert#1\Bigr\Vert}
\newcommand{\biggnorm}[1]{\biggl\Vert#1\biggl\Vert}

\newtheorem{theorem}{Theorem}[section]
\newtheorem{lemma}[theorem]{Lemma}
\newtheorem{corollary}[theorem]{Corollary}
\newtheorem{proposition}[theorem]{Proposition}
\newtheorem{definition}[theorem]{Definition}
\theoremstyle{remark}
\newtheorem{remark}[theorem]{\bf Remark}
\theoremstyle{definition}

\numberwithin{equation}{section}

\author{Léonard Cadilhac}
\email{cadilhac@imj-prg.fr}
\address{IMJ-PRG – UMR7586
Sorbonne University, France}
\author[C. Le Merdy]{Christian Le Merdy}
\email{clemerdy@univ-fcomte.fr}
\address{Laboratoire de Math\'ematiques de Besan\c con, UMR 6623,
CNRS, Universit\'e Marie et Louis Pasteur,
25030 Besan\c{c}on Cedex, France}
\author[S. Zadeh]{Safoura Zadeh}
\email{jsafoora@gmail.com}
\address{Institute of Mathematics, University of Bristol, U.K.}

\begin{document}

\title[Random ergodic averages]{A Non-commutative Individual Ergodic Theorem Along Sparse Random Subsequences}

\begin{abstract}
Let $(M,\tau)$ be a semifinite von Neumann algebra, let $J$ be a 
trace-preserving Jordan isomorphism, and let $(n_k)_{k\geq 1}$ be a random increasing sequence of integers obtained by selecting each integer $n\geq 1$
independently with probability $n^{-\alpha}$, where $0<\alpha<\frac12$. We show that, almost surely, for every $x\in L^1(M)$,
\(
   \frac1m\sum_{k=1}^m J^{n_k}(x)
\)
converges bilaterally almost uniformly. This extends LaVictoire's classical random $L^1$ ergodic theorem to the non-commutative setting. 
\end{abstract}

\maketitle

\noindent
{\it 2020 Mathematics Subject Classification: Primary 46L53, 46L55; Secondary 46L51, 37A99}

\smallskip
\noindent
{\it Key words:}  random sets; individual ergodic theorems; non-commutative \(L^p\)-spaces.

\bigskip
\section{Introduction}\label{1}
Let $(X,\mathcal F,\mu)$ be a measure space and let $T\colon
X\to X$ be a
measure-preserving transformation. Birkhoff's individual ergodic theorem
asserts that, for every $f\in L^1(X)$, the Ces\`aro averages
\[
   \frac{1}{N}\sum_{n=1}^N f\circ T^n
\]
converge almost everywhere as $N$ tends to infinity. A modern proof relies on
a weak-type $(1,1)$-maximal inequality, which is then combined with density
arguments to pass from a convenient dense class of functions to all of
$L^1$; see, for example, \cite[Chapter~1]{Kren}.

A non-commutative analogue of this circle of ideas was 
initiated by Yeadon
\cite{Ye}, after Lance's pioneering work \cite{La}. 
In this setting, one replaces a measure space by a von Neumann
algebra with a normal, faithful and semifinite trace $(M,\tau)$, the classical
spaces $L^p(X)$ by the associated non-commutative spaces $L^p(M)$, and
pointwise almost everywhere convergence by bilaterally almost uniform
convergence (b.a.u.\ convergence in short). 
If $T\colon M\to M$ is a positive Dunford--Schwartz operator $T$ on $M$ (that is, a positive operator which is
contractive on both $M=L^\infty(M)$ and $L^1(M)$),
Yeadon's maximal inequality provides the non-commutative counterpart of the
classical weak-type $(1,1)$-maximal inequality and yields individual ergodic
convergence in $L^1(M)$. Subsequent developments, notably the
non-commutative maximal ergodic theorems of Junge and Xu \cite{JX},
established the corresponding maximal inequalities and individual ergodic
theorems in non-commutative $L^p$-spaces for $1<p<\infty$, and provided a
systematic framework for b.a.u.\ convergence. Further individual ergodic results in the semifinite setting can be found in \cite{CL}.

A natural direction is to replace the full sequence of integers by a sparse
subsequence. In the commutative theory, this question asks whether
\[
   \frac{1}{m}\sum_{k=1}^m f\circ T^{a_k}
\]
converges almost everywhere for a prescribed increasing sequence
$(a_k)_{k\geq 1}$ of positive integers. Bourgain's pointwise ergodic theorem along
polynomial sequences, in particular along the squares ($a_k = k^2$), \cite{B2}, established a
fundamental positive result for $L^p$, $p>1$. The corresponding non-commutative polynomial ergodic theorem was subsequently established by Hong and Wang \cite{HW}, who obtained b.a.u.\ convergence for the associated non-commutative ergodic averages in the range \(1<p<\infty\). The range
$p=1$ is dramatically different: the square averages are universally
``$L^1$-bad'', meaning that for every invertible measure-preserving transformation on a non-atomic finite measure space there exists a function $f\in L^1$ for which the averages along the squares fail to converge
almost everywhere, as shown by Buczolich and Mauldin
\cite{BM1, BM2}. Thus, sparse subsequences expose a genuine endpoint
obstruction which is invisible in the reflexive range $1<p<\infty$.

Sparse random subsequences provide a flexible probabilistic model for this
phenomenon. Fix $0<\alpha<1$, and let $(X_n)_{n\geq 1}$ be independent
Bernoulli random variables satisfying
\[
   \mathbb P(X_n=1)=n^{-\alpha}.
\]
The corresponding random set has, almost surely, cardinality of order
$N^{1-\alpha}$ up to $N$. More precisely, if $(n_k)_{k\geq 1}$ denotes the
sequence of selected integers, then
$n_k\sim ((1-\alpha)k)^{\frac{1}{1-\alpha}}$
almost surely. The case $\alpha=\frac12$ therefore provides 
the same order of growth
as the squares, while $\alpha<\frac12$ gives a random sequence 
denser than the
squares. For $p>1$, random ergodic theorems 
for averages of the form (\ref{Kind}) below are by now well
established in both classical and non-commutative settings; 
see, for example,
\cite{B, LZ}. More specifically,
in the non-commutative framework, the recent result of
Le Merdy and Zadeh \cite{LZ} shows that, for every positive
Dunford--Schwartz operator $T$ on $M$ and every
$1<p<\infty$, the randomly selected averages
\begin{equation}\label{Kind}
   \frac{1}{m}\sum_{k=1}^m T^{n_k(\omega)}(x)
\end{equation}
converge bilaterally almost uniformly for all $x\in L^p(M)$, almost surely
in $\omega$.

The endpoint $p=1$ is considerably more delicate. In the commutative case,
LaVictoire, \cite{LaV}, proved an $L^1$ individual ergodic theorem for the above sparse
random subsequences in the range $0<\alpha<\frac12$. 
His result may be
viewed as a random $L^1$ substitute for the square problem, valid for
sequences which are almost surely slightly denser than the squares. The
corresponding question at and beyond the square-density threshold, namely for
$\alpha\geq \frac12$, remains outside the reach of the known $L^1$ random
methods. The proof of LaVictoire's theorem relies on a refined
Calder\'on--Zygmund decomposition together with probabilistic estimates for
some random kernels.

The purpose of the present paper is to establish the endpoint \(L^1\) individual ergodic theorem for sparse random subsequences in the non-commutative setting. We obtain a b.a.u.\ convergence theorem for trace-preserving Jordan isomorphisms, thereby extending the random \(L^1\) phenomenon of LaVictoire to non-commutative framework. More precisely, let $(M,\tau)$ be a tracial von Neumann
algebra and let $J\colon M\to M$ be a trace-preserving Jordan isomorphism. Let
$0<\alpha<\frac12$, let $(X_n)_{n\geq 1}$ be as above and
set
\[
   W_N=\sum_{n=1}^N n^{-\alpha}.
\]
Our main result asserts that there exists
a measurable set $\Omega'\subset\Omega$ of 
probability one such that, for
every $\omega\in\Omega'$ and every $x\in L^1(M)$,
\[
   \frac{1}{W_N}\sum_{n=1}^N X_n(\omega)J^n(x)
   \longrightarrow P_1(x)
   \qquad\text{b.a.u.\ as }N\to\infty,
\]
where $P_1$ denotes the $L^1$ extension of the ergodic projection
associated with $J$ on $L^2(M)$. Equivalently, 
\[
   \frac{1}{m}\sum_{k=1}^m J^{n_k(\omega)}(x)
   \longrightarrow P_1(x)
   \qquad\text{b.a.u.\ as }m\to\infty.
\]
This gives an almost sure individual random ergodic theorem in
non-commutative $L^1$ for sparse random subsequences of density
$N^{1-\alpha}$, $0<\alpha<\frac12$.

The restriction to Jordan isomorphisms is natural for the method used here.
Trace-preserving Jordan isomorphisms act isometrically on $L^1(M)$ and
preserve the projection structure needed for transference. In particular,
they include the usual trace-preserving $\ast$-automorphisms, while still
allowing the non-multiplicative Jordan setting. The proof reduces the desired
maximal inequality for the random ergodic averages associated with $J$ to a
maximal inequality for convolution operators acting on
$\ell^1_{\mathbb Z}(L^1(M))$. This transference step is one of the points at
which the Jordan structure is essential.

A principal difficulty lies in proving the required weak-type $(1,1)$-maximal 
inequality in the non-commutative endpoint space 
$\ell^1_{\mathbb Z}(L^1(M))$. In the classical
proof, LaVictoire uses a Calder\'on--Zygmund decomposition adapted to the
random kernels. Such a decomposition cannot be transferred directly to the
non-commutative setting: projections do not commute with the functions being
decomposed, there are no pointwise level sets in the classical sense, and the
bad part of the decomposition must be controlled bilaterally.
Fortunately, non-commutative extensions of Calder\'on--Zygmund decomposition 
and several of its applications have been studied over the last decade (see for example \cite{P,CCAP,HLX}). 
For our argument, we use and adapt the construction described in \cite{CCAP}, which is based on
Cuculescu projections \cite{C} and dyadic conditional expectations.

The argument separates the random kernel into its deterministic mean and its
centred part. The deterministic contribution is controlled by Yeadon's
maximal inequality, together with summation by parts, while the centred
contribution is handled through a non-commutative version of LaVictoire's  criterion. The hypothesis $\alpha<\frac12$ enters precisely in
the almost sure off-diagonal correlation estimate for the centred kernels.
Once the weak-type $(1,1)$-maximal inequality has been obtained, the transference principle carries it to the iterates of $J$; the
$L^2$ convergence theorem of \cite{LZ} and an appropriate 
non-commutative Banach
principle then yield the asserted $L^1$ b.a.u.\ convergence.

The paper is organized as follows. Section 2 recalls the necessary background
on non-commutative $L^p$-spaces, Dunford--Schwartz operators, b.a.u.\ convergence, and sparse random sets. Section 3 proves a
non-commutative Banach principle for b.a.u.\ convergence and develops the
transference principle which carries weak-type $(1,1)$-maximal inequality from the
shift on $\ell^1_{\mathbb Z}(L^1(M))$ to trace-preserving Jordan
isomorphisms on $L^1(M)$. Section 4 is the technical core of the paper: it
establishes a criterion for weak-type $(1,1)$-maximal 
inequalities for the
random convolution averages, using a non-commutative
Calder\'on--Zygmund decomposition as indicated before. 
Finally, Section 5
combines these ingredients with the known 
$L^2$ random ergodic theorem of
\cite{LZ} to prove the desired $L^1$ b.a.u.\ convergence.\\

\section{Basic definitions and background}\label{1}

Let $M$ be a semifinite von Neumann algebra, with positive cone $M_+$,
and let $\tau\colon M_+\to[0,\infty]$ be a fixed normal semifinite
faithful trace. We refer to $(M,\tau)$ as a tracial von Neumann algebra.
Assume that $M$ acts on a Hilbert space $H$, and denote by $L^0(M)$ the
$*$-algebra of all closed, densely defined $\tau$-measurable operators on
$H$. For $1\leq p<\infty$, we consider the non-commutative $L^p$-space
$$
L^p(M)=\bigl\{x\in L^0(M)\,:\,\tau\bigl(\vert x\vert^p\bigr)<\infty\bigr\}.
$$
It is a Banach space for the norm
$$
\norm{x}_p=(\tau(\vert x\vert^p))^{\frac1p}.
$$
Moreover, $x^*\in L^p(M)$ and $\norm{x^*}_p=\norm{x}_p$ for every
$x\in L^p(M)$. We set $L^\infty(M)=M$. We also use the notation $\tau$ 
for the natural extension of the trace to the positive cone
$L^0(M)_+$  for the contractive functional
on $L^1(M)$ extending $\tau_{\vert L^1(M)_+}$.
We refer to \cite{Hiai,PX} for further
details and background.

The duality between these spaces is given by the non-commutative
H\"older inequality. For $1\leq p<\infty$, let
$p'=\frac{p}{p-1}$ denote the conjugate exponent. Then
$L^p(M)\,\cdotp L^{p'}(M)=L^1(M)$ and
\begin{equation}\label{2H}
\vert\tau(xy)\vert\leq \norm{x}_p\norm{y}_{p'},
\qquad x\in L^p(M),\ y\in L^{p'}(M).
\end{equation}
The pairing
\[
\langle x,y\rangle:=\tau(xy)
\]
induces an isometric identification
$L^{p'}(M)\simeq L^p(M)^*$. In particular, $L^p(M)$ is reflexive for
$1<p<\infty$. The case $p=2$ gives the Hilbert space structure
$$
(x\vert y)=\tau(y^*x),\qquad x,y\in L^2(M).
$$
The $L^p$-spaces also carry their natural $M$-bimodule structure: for
$1\leq p\leq\infty$,
$$
M\cdotp L^p(M)\,\cdotp M=L^p(M),
$$
with
$$
\norm{y_1xy_2}_p
\leq\norm{y_1}_\infty\norm{x}_p\norm{y_2}_\infty,
\qquad x\in L^p(M),\ y_1,y_2\in M.
$$

We will also use the compatibility of the $L^p$-scale with
approximation and interpolation. For any $1\leq p<\infty$, the space
$L^p(M)\cap L^q(M)$ is dense in $L^p(M)$ for every
$1\leq q\leq\infty$. Moreover,
$$
L^p(M)=[L^\infty(M),L^1(M)]_{\frac{1}{p}},
$$
isometrically, where the right-hand side is defined via complex
interpolation.

We now turn to the operators that generate the ergodic averages.
Let $T\colon M\to M$ be a bounded linear map. We say that $T$ is positive
if $T(M_+)\subset M_+$, and that $T$ is a Dunford--Schwartz operator if
$$
\norm{T(x)}_\infty\leq \norm{x}_\infty\ 
\forall x\in M
\qquad \hbox{and}\qquad
\norm{T(x)}_1\leq \norm{x}_1\
\forall x\in M\cap L^1(M),
$$
for all $x\in M$. The two endpoint estimates allow $T$ to be extended
consistently across the entire $L^p$-scale: by density and interpolation,
its restriction to $M\cap L^p(M)$ has a unique contractive extension
$$
T_p\colon L^p(M)\longrightarrow L^p(M),
\qquad 1\leq p\leq\infty.
$$

For positive contractions, the trace gives an equivalent description of
the Dunford--Schwartz property. More precisely, if $T$ is positive, then
$T$ is a Dunford--Schwartz operator if and only if
$\tau\circ T\leq\tau$, see \cite[Lemma 1.1]{JX}.

The corresponding deterministic ergodic averages are
\begin{equation}\label{2Cesaro}
A_N(T)=\frac1N\sum_{n=1}^N T^n,\qquad N\geq1.
\end{equation}
When $1<p<\infty$, the reflexivity of $L^p(M)$ allows the Mean Ergodic
Theorem to be applied to $T_p$. We have the decomposition
$$
L^p(M)= {\rm Ker}(I_{L^p} -T_p)\oplus
\overline{\rm Ran}(I_{L^p}-T_p).
$$
Let $P_p$ denote the projection onto
${\rm Ker}(I_{L^p} -T_p)$ along
$\overline{\rm Ran}(I_{L^p}-T_p)$. Then
$P_p$ is the strong limit of the Ces\`aro averages
$(A_N(T_p))_{N\geq1}$. See, e.g.,
\cite[Section 8.4]{EFHN} or \cite[Theorem 1.15]{LM}. 

\begin{lemma}\label{2Gamma1}
Let $T\colon M\to M$ be a positive 
Dunford--Schwartz operator. The restriction of 
$P_2$ to $L^1(M)\cap L^2(M)$
extends to a contraction 
$$
P_1\colon L^1(M)\longrightarrow L^1(M).
$$
\end{lemma}

\begin{proof}
Since $\tau\circ T\leq \tau$, we have
$\tau\circ A_N(T)\leq\tau$ for all $N\geq1$. Since each $A_N(T)$ is
positive and
\[
P_2(x)=L^2\text{-}\lim_N[A_N(T_2)](x)
\qquad x\in L^2(M),
\]
Fatou's property for $\tau$ gives \(\tau\circ P_2\leq\tau\)
on $L^1(M)_+\cap L^2(M)$. Arguing as in
\cite[Lemma 1.1]{JX}, we obtain the contractive extension
$P_1\colon L^1(M)\to L^1(M)$.
\end{proof}

Projections in $M$ will play an important role in the convergence
arguments. We denote by $\P(M)$ the set of projections in $M$, and for
$e\in\P(M)$ write
$e^\perp=1-e$ for its complement. We shall repeatedly use the estimate
\begin{equation}\label{2Tr}
\tau\bigl((e\wedge f)^\perp\bigr)
\leq \tau(e^\perp)+\tau(f^\perp),
\end{equation}
valid for all $e,f\in\P(M)$. We refer to
\cite[Section 5.1]{Tak} for general background on projections in von
Neumann algebras.

Following \cite[Definition 6.1]{JX}, a sequence
$(x_n)_{n\geq1}$ in $L^0(M)$ is said to converge bilaterally almost
uniformly (b.a.u.\ in short) to $x\in L^0(M)$ if, for every $\varepsilon>0$, there
exists $e\in\P(M)$ such that
$$
\tau(e^\perp)\leq\varepsilon
\qquad\text{and}\qquad
\|e(x_n-x)e\|_\infty\longrightarrow0
\quad\text{as }n\to\infty.
$$

We record a simple observation.

\begin{remark}\label{2tn}
For any $x\in L^{0}(M)$ and any sequence $(t_n)_{n\geq1}$ of complex
numbers satisfying $\lim_n t_n=1$, the sequence $(t_nx)_{n\geq1}$
converges b.a.u.\ to $x$. Indeed, given $\varepsilon>0$, the
$\tau$-measurability of $x$ provides $e\in\P(M)$ such that
$\tau(e^\perp)\leq\varepsilon$ and $exe\in M$.
\end{remark}

We finally introduce the random sampling scheme. Following \cite{LZ},
and also \cite{FLW} and \cite{LaV}, let $(\Omega,\mathbb P)$ be a
probability space, fix $\alpha\in(0,1)$, and let
$(X_n)_{n\geq1}$ be independent Bernoulli random variables
$X_n\colon\Omega\to\{0,1\}$ satisfying
\begin{equation}\label{2Alpha}
\mathbb P(X_n=1)=n^{-\alpha},\qquad n\geq1.
\end{equation}
The associated normalization is
\begin{equation}\label{2WN}
W_N:=\sum_{n=1}^N n^{-\alpha}.
\end{equation}
Its asymptotic behaviour is
\begin{equation}\label{2WNEquiv}
W_N\sim(1-\alpha)^{-1}N^{1-\alpha}.
\end{equation}
The strong law of large numbers \cite[Appendix A]{FLW} shows that for almost every $\omega\in\Omega$,
\begin{equation}\label{2SLLN}
\sum_{n=1}^NX_n(\omega)\sim W_N,
\end{equation}

We associate to the random
sequence $(X_n)$ the hitting-time random
variables $n_k\colon\Omega\to\Ndb$, $k\geq1$, defined almost everywhere by
$$
n_k=\min\bigl\{N\geq1\,:\,X_1+\cdots+X_N=k\bigr\}.
$$
The preceding counting asymptotic yields
$$
n_k(\omega)\sim
\bigl((1-\alpha)k\bigr)^{\frac{1}{1-\alpha}}
$$
for almost every $\omega\in\Omega$.

The random ergodic theorem of \cite{LZ} identifies the almost sure
behaviour of the corresponding sparse averages. We recall the result in
the form that will be used below. Note that the two parts (1) and (2)
are essentially equivalent, see the proof of \cite[Theorem 5.2]{LZ} for details.

\begin{theorem}\label{2Back}
Let $T\colon M\to M$ be a positive Dunford--Schwartz operator, and let
$1<p<\infty$. Then there exists a measurable set
$\Omega'\subset\Omega$ with $\Pdb(\Omega')=1$ such that, for every
$\omega\in\Omega'$ and every $x\in L^p(M)$, the following equivalent statments hold:
\begin{itemize}
\item[(1)]
$$
\frac{1}{W_N}\sum_{n=1}^N X_n(\omega)\, T_p^n(x)
\,\longrightarrow P_p(x)\quad\hbox{b.a.u.}
\quad\hbox{as } N\to\infty.
$$
\item[(2)]
$$
\frac{1}{m}\sum_{k=1}^m T_p^{n_k(\omega)}(x)
\,\longrightarrow P_p(x)\quad\hbox{b.a.u.}
\quad\hbox{as } m\to\infty.
$$
\end{itemize}
\end{theorem}

\section{Preliminary results}\label{3}

\subsection{A Banach principle for b.a.u.\ convergence}

In this subsection, we establish a non-commutative version of the classical Banach principle for pointwise convergence; see \cite[Chapter 1, Theorem 7.2]{Kren}. More precisely, we show that, under a weak-type $(1,1)$-maximal inequality, b.a.u.\ convergence on an $L^1(M)$-dense subset extends to the whole space. We then apply this principle to weighted ergodic averages. We refer to \cite{Ye}, \cite[Section 6]{JX}, and \cite{CL} for related results.

\begin{definition}\label{3MI}
Let \((V_j)_{j\ge1}\) be a sequence of bounded 
maps \(V_j : L^1(M) \longrightarrow L^1(M)\).
We say that \((V_j)_{j\ge1}\) satisfies a weak-type
$(1,1)$-maximal inequality if there is a constant \(C\geq 0\) 
such that for every \(x\in L^1(M)\) and every \(\lambda>0\), 
there exists a projection \(e\in \P(M)\) satisfying
\begin{equation}\label{3Weak11}
\tau(e^\perp)\le \frac{C\|x\|_1}{\lambda}
\qquad\text{and}\qquad
\|eV_j(x)e\|_\infty \le \lambda, \quad \forall j\ge1.
\end{equation}
\end{definition}

\begin{remark}\label{3Pos}
It is enough to verify the weak-type $(1,1)$-maximal inequality in Definition \ref{3MI} for positive elements. Indeed, let $x\in L^1(M)$. We can write $x=(x_1-x_2)+i(x_3-x_4)$, where $x_1,\ldots,x_4\in L^1(M)_+$ satisfy $\|x_i\|_1\leq \|x\|_1$ for $i=1,\ldots,4$. Suppose that for some $C\geq0$, the estimate in \eqref{3Weak11} is valid for every positive $x\in L^1(M)$. Applying this estimate to each $x_i$ with $\lambda/4$ in place of $\lambda$, we obtain projections $e_1,\ldots,e_4\in\P(M)$ such that $\tau(e_i^\perp)\leq 4\, C\lambda^{-1}\|x\|_1$ and $\|e_iV_j(x_i)e_i\|_\infty\leq\lambda/4$ for every $i\in\{1,\ldots,4\}$ and for all $j\geq 1$. Set $e=e_1\wedge\cdots\wedge e_4$. By \eqref{2Tr}, we have $\tau(e^\perp)\leq 16\, C\lambda^{-1}\norm{x}_1$. Moreover $\norm{eV_j(x)e}_\infty\leq\lambda$ for
all $j\geq 1$. Hence, $(V_j)_{j\geq1}$ satisfies a weak-type $(1,1)$-maximal inequality with constant $16\, C$.
\end{remark}

The following result is a non-commutative Banach principle that we will use to extend b.a.u.\ convergence from an $L^1(M)$-dense subset to all of $L^1(M)$ in Corollary \ref{3Y2}.
\begin{lemma}\label{3Y1}
Let $(V_j)_{j\geq1}$ be a sequence of bounded maps on $L^1(M)$ that
satisfies a weak-type $(1,1)$-maximal inequality. Let $x\in L^1(M)$,
and suppose that there exists a sequence $(y_m)_{m\geq1}$ in $L^1(M)$
such that
\[
\|x-y_m\|_1\longrightarrow0
\qquad\text{as }m\to\infty,
\]
and for every $m\geq1$,
\[
V_j(y_m)\longrightarrow0
\qquad\text{b.a.u.\ as }j\to\infty.
\]
Then
\[
V_j(x)\longrightarrow0
\qquad\text{b.a.u.\ as }j\to\infty.
\]
\end{lemma}

\begin{proof}
Let $C\geq0$ be a constant for the weak-type $(1,1)$-maximal inequality in Definition \ref{3MI}, and fix $\varepsilon>0$. Passing to a subsequence, 
we may assume that
$\|x-y_m\|_1\leq4^{-m}\varepsilon$,
for all $m\geq1$.
Set $z_m=x-y_m$, so that $x=y_m+z_m$ and
\[
\|z_m\|_1\leq4^{-m}\varepsilon,
\qquad m\geq1.
\]
Applying the weak-type $(1,1)$-maximal inequality to $z_m$ with $\lambda=2^{-m}$, we obtain projections $e_m\in\P(M)$ such that
\(\tau(e_m^\perp)\leq C\varepsilon 2^{-m}\)
and
\[
\|e_mV_j(z_m)e_m\|_\infty\leq2^{-m},
\qquad j\geq1.
\]
Let
\(e=\bigwedge_{m=1}^\infty e_m.
\)
By \eqref{2Tr},
\[
\tau(e^\perp)
\leq\sum_{m=1}^\infty\tau(e_m^\perp)
\leq C\varepsilon.
\]

On the other hand, since for every $m\geq1$, $V_j(y_m)\to0$ b.a.u.\ as $j\to\infty$, for each $m\geq1$ there exists a projection $f_m\in\P(M)$ such that
\[
\tau(f_m^\perp)\leq\varepsilon2^{-m}
\]
and
\[
\|f_mV_j(y_m)f_m\|_\infty\longrightarrow0
\qquad\text{as }j\to\infty.
\]
Set
\(f=\bigwedge_{m=1}^\infty f_m
\ \text{and}\
g=e\wedge f.\)
Again by \eqref{2Tr},
\(\tau(f^\perp)
\leq\sum_{m=1}^\infty\tau(f_m^\perp)
\leq\varepsilon, \)
and hence
\[
\tau(g^\perp)
\leq\tau(e^\perp)+\tau(f^\perp)
\leq(1+C)\varepsilon.
\]

It remains to show that
\(\|gV_j(x)g\|_\infty\longrightarrow0
\ \text{as }j\to\infty.\)
Fix $m\geq1$. Since $x=y_m+z_m$, the triangle inequality gives
\[
\|gV_j(x)g\|_\infty
\leq
\|gV_j(y_m)g\|_\infty
+
\|gV_j(z_m)g\|_\infty.
\]
Since $g\leq f_m$ and $g\leq e_m$, we obtain
\[
\begin{aligned}
\|gV_j(x)g\|_\infty
&\leq
\|f_mV_j(y_m)f_m\|_\infty
+
\|e_mV_j(z_m)e_m\|_\infty\\
&\leq
\|f_mV_j(y_m)f_m\|_\infty+2^{-m}.
\end{aligned}
\]
Therefore,
\[
\limsup_{j\to\infty}\|gV_j(x)g\|_\infty
\leq2^{-m}.
\]
Since $m\geq1$ was arbitrary, letting $m\to\infty$ yields \(\lim_{j\to\infty}\|gV_j(x)g\|_\infty = 0\), which completes the proof.
\end{proof}

We now apply the preceding Banach principle to weighted ergodic averages. We retain the notation from Lemma \ref{2Gamma1}.

\begin{corollary}\label{3Y2}
Let $T\colon M\to M$ be a positive Dunford--Schwartz operator. Let
$d=(d_n)_{n\geq1}$ be a non-zero sequence of non-negative real numbers,
and let $n_0\geq1$ be such that $d_{n_0}\neq0$. Define the weighted
averages
\[
A_N^{(d)}(T)=\frac{1}{W_N}\sum_{n=1}^N d_nT^n,
\qquad
W_N^{(d)}=\sum_{n=1}^N d_n,
\qquad N\geq n_0.
\]
Assume that $(A_N^{(d)}(T_1))_{N\geq n_0}$ satisfies a weak-type
$(1,1)$-maximal inequality and that
\[
[A_N^{(d)}(T_2)](y)\longrightarrow P_2(y)
\qquad\text{b.a.u.\ as }N\to\infty
\]
for all $y\in L^2(M)$. Then, for every $x\in L^1(M)$,
\[
[A_N^{(d)}(T_1)](x)\longrightarrow P_1(x)
\qquad\text{b.a.u.\ as }N\to\infty.
\]
\end{corollary}

\begin{proof}
Let $x\in L^1(M)$. By the density of $L^1(M)\cap L^2(M)$ in $L^1(M)$, choose a sequence $(z_m)_{m\geq1}$ in $L^1(M)\cap L^2(M)$ such that $\|x-z_m\|_1\to0$ as $m\to\infty$. Since $P_1$ is contractive and agrees with $P_2$ on $L^1(M)\cap L^2(M)$, we have $\|P_1(x)-P_2(z_m)\|_1\leq\|x-z_m\|_1\to0$. Set
\[
y_m = z_m - P_2(z_m),\qquad m\geq 1.
\]
Then $y_m\in L^1(M)\cap L^2(M)$, $P_2(y_m)=0$, and
\[
\|(x-P_1(x))-y_m\|_1
\leq\|x-z_m\|_1+\|P_1(x)-P_2(z_m)\|_1
\longrightarrow0 \qquad \text{as } m\to\infty.
\]
By hypothesis, \[A_N^{(d)}(T_1)(y_m)=A_N^{(d)}(T_2)(y_m)\longrightarrow P_2(y_m)=0
\qquad\text{b.a.u.\ as }N\to\infty.\] Applying Lemma~\ref{3Y1}, we obtain
\[
A_N^{(d)}(T_1)(x-P_1(x))\longrightarrow0
\qquad\text{b.a.u.\ as }N\to\infty.
\]
Finally, by the definition of $P_1$,
we have  \(T_1 P_1(x)=P_1(x)\), which implies 
that 
$$
A_N^{(d)}(T_1)(P_1(x))=P_1(x),
\qquad N\geq n_0.
$$
Consequently, 
\[
A_N^{(d)}(T_1)(x)\longrightarrow P_1(x)
\qquad\text{b.a.u.\ as }N\to\infty.\qedhere
\]
\end{proof}

\subsection{A transference principle for Jordan isomorphisms}\label{sec:transference}

We now establish the transference principle that will allow us to pass the weak-type $(1,1)$-maximal inequalities obtained for the shift on $\ell^1_{\mathbb Z}(L^1(M))$ to the iterates of a trace-preserving Jordan isomorphism on $L^1(M)$. The argument follows the non-commutative transference framework of \cite{HLW}.

Set
\(\mathcal N:=\ell^\infty_{\mathbb Z}(M)
\)
and equip $\mathcal N$ with the normal semifinite faithful trace
\begin{equation}\label{2Tau-Infty}
\tau_\infty\bigl((x(k))_{k\in\mathbb Z}\bigr)
=\sum_{k\in\mathbb Z}\tau(x(k)),
\qquad (x(k))_{k\in\mathbb Z}\in\mathcal N_+.
\end{equation}
Then, for every $1\leq p\leq\infty$,
\[
L^p(\mathcal N)=\ell^p_{\mathbb Z}(L^p(M))
\quad\hbox{isometrically}.
\]
The shift
\[
S\colon\mathcal N\to\mathcal N,\qquad
S\bigl((y(k))_{k\in\mathbb Z}\bigr)
=(y(k-1))_{k\in\mathbb Z},
\]
is a positive Dunford--Schwartz operator. We will transfer estimates for $S$ to a Jordan isomorphism. Recall that a bounded map $J\colon M\to M$ is a Jordan homomorphism if
\[
J(x^*)=J(x)^*,\qquad J(x^2)=J(x)^2,\qquad x\in M.
\]
We refer to \cite{HOS} for the basic properties of such maps. In particular, every Jordan homomorphism is a positive contraction and satisfies
\begin{equation}\label{3Jordan}
J(yxy)=J(y)J(x)J(y),\qquad x,y\in M.
\end{equation}
Moreover, every injective Jordan homomorphism is an isometry, and a bijective Jordan homomorphism is called a Jordan isomorphism.

In the sequel, fix a trace-preserving Jordan isomorphism
\(J\colon M\to M.
\)
Its inverse is again a trace-preserving Jordan homomorphism. Hence both $J$ and $J^{-1}$ are Dunford--Schwartz operators, and their $L^1$-extensions $J_1$ and $J_1^{-1}$ are contractions. It follows that
\(J_1\colon L^1(M)\to L^1(M)\)
is an onto isometry.

Given a sequence $(c_n)_{n\ge1}$ of non-negative real numbers, consider
\[
V_N=\sum_{n=1}^N c_nS^n,
\qquad
V'_N=\sum_{n=1}^N c_nJ^n,
\qquad N\ge1.
\]
These operators extend to $\ell^1_{\mathbb Z}(L^1(M))=L^1(\mathcal N)$ and $L^1(M)$, respectively, and we use the same notation for their $L^1$-extensions.

\begin{lemma}[Transference principle]\label{lem:transference}
If the sequence $(V_N)_{N\ge1}$ satisfies a weak-type $(1,1)$-maximal inequality, then the sequence $(V'_N)_{N\ge1}$ also satisfies a weak-type $(1,1)$-maximal inequality.
\end{lemma}

\begin{proof}
Let $C\ge0$ be a constant for the weak-type $(1,1)$-maximal inequality of $(V_N)_{N\ge1}$. By \cite[Lemma 2.3]{HLW}, it is enough to prove the corresponding estimate uniformly for every finite family
\((V'_1,\ldots,V'_{N_0}),
\)
with a constant independent of $N_0$. Fix $N_0\ge1$ and set
\[
K=\{1,\ldots,N_0\},
\qquad
F=\{-m,\ldots,m\},
\]
where $m$ is chosen so that
\(
\frac{|F-K|}{|F|}\le2.
\)
For $x\in L^1(M)_+$ and $\lambda>0$, define
\[
y=(y(j))_{j\in\mathbb Z}\in\ell^1_{\mathbb Z}(L^1(M)),
\qquad
y(j)=\chi_{F-K}(j)J^{-j}(x).
\]
Since $J_1$ is an isometry,
\begin{equation}\label{3L1}
\|y\|_1
=\sum_{j\in F-K}\|J^{-j}(x)\|_1
=|F-K|\, \|x\|_1.
\end{equation}
For $k\in F$ and $1\le N\le N_0$, we have
\[
\begin{aligned}
[V_N(y)](k)=\sum_{n=1}^N c_n y(k-n)=\sum_{n=1}^N c_nJ^{-(k-n)}(x)=\sum_{n=1}^N c_nJ^{n-k}(x)=J^{-k}(V'_N(x)).
\end{aligned}
\]
Thus, on \(F\), \(V'_N(x)\) is recovered from the corresponding coordinates of \(V_N(y)\) via \(J^k\). Applying the weak-type $(1,1)$-maximal inequality for $(V_N)_{N\ge1}$ to $y$, we obtain a projection
\[
e=(e(j))_{j\in\mathbb Z}\in\mathcal P(\mathcal N)
\]
such that
\[
\tau_\infty(e^\perp)\le\frac{C\|y\|_1}{\lambda},
\qquad
\|eV_N(y)e\|_\infty\le\lambda,
\quad N\ge1.
\]
Choose $k_0\in F$ so that $\tau(e(k_0)^\perp)$ is minimal. Then

\[
\tau(e(k_0)^\perp)\le \frac1{|F|}\sum_{k\in F}\tau(e(k)^\perp)
\le \frac1{|F|}\tau_\infty(e^\perp)
\le \frac{C|F-K|}{\lambda|F|}\|x\|_1
\le \frac{2C}{\lambda}\|x\|_1.
\]

As $J$ maps projections to projections, we have \(J^{k_0}(e(k_0))\in\mathcal P(M)\). Set \(f=J^{k_0}(e(k_0))\).
The trace-preserving property of $J$ gives
\[
\tau(f^\perp)\le\frac{2C}{\lambda}\|x\|_1.
\]

It remains to transfer the uniform bound. Using the triple homomorphism property \eqref{3Jordan} and the identity above,
\begin{align*} 
f V'_N(x) f &= J^{k_0}(e(k_0)) J^{k_0}J^{-k_0} (V'_N(x)) J^{k_0}(e(k_0)) \\ &= J^{k_0}\bigl(e(k_0) J^{-k_0}(V'_N(x)) e(k_0)\bigr) \\ &= J^{k_0}\bigl(e(k_0) (V_N(y))(k_0) e(k_0)\bigr) \\ &= J^{k_0}\bigl([e V_N(y) e](k_0)\bigr). 
\end{align*}
Since $J^{k_0}$ is an isometry on $M$,
 \[ \|f V'_N(x) f\|_\infty = \|[e V_N(y) e](k_0)\|_\infty \le \|e V_N(y) e\|_\infty \le \lambda. \]
Therefore, $(V'_N)_{1\le N\le N_0}$ satisfies the weak-type $(1,1)$-maximal inequality with constant $2C$, uniformly in $N_0$. 
\end{proof}

\section{A criterion for weak-type
$(1,1)$-maximal inequality}\label{4}

The aim of this section is to establish a non-commutative analogue of
\cite[Proposition 3.1]{LaV}, stated below as Proposition \ref{4Main}.
We first develop a Calder\'on-Zygmund type decomposition adapted to the
setting of this paper, and then establish a non-commutative $L^2$-space
version of \cite[Lemma 4.2]{LaV}.

\subsection{A Calderón-Zygmund type decomposition}\label{CZ}

As before, we fix a tracial von Neumann algebra $(M,\tau)$ and set
$$
\N=\ell^\infty_{\mathbb Z}(M).
$$
We equip $\N$ with the normal semifinite faithful trace $\tau_\infty$
defined in (\ref{2Tau-Infty}).

For $s\in\Ndb_0$ and $k\in\Zdb$, define the dyadic intervals
$$
Q_{s,k}
=\{k2^s,\ldots,(k+1)2^s-1\}
=[k2^s,(k+1)2^s)\cap\Zdb.
$$
For each $s$, the family $(Q_{s,k})_{k\in\mathbb Z}$ is a partition of
$\Zdb$, and each $Q_{s,k}$ contains $2^s$ elements.

Following \cite[Section 2.3]{HLX}, for every $s\geq0$, let $\N_s$ be the
space of all $x=(x(m))_{m\in\mathbb Z}$ in $\N$ that are constant on
each $Q_{s,k}$; that is,
$x(m)=x(m')$ whenever there exists $k\in\mathbb Z$ such that
$m,m'\in Q_{s,k}$.
Each $\N_s$ is a sub-von Neumann algebra of $\N$, with
$\N_0=\N$ and $\N_{s+1}\subset\N_s$ for all $s\geq0$. Thus,
$(\N_s)_{s\geq0}$ is a reversed filtration of $\N$.
Let $(\Edb_s)_{s\geq0}$ be the associated conditional expectations.
For any $x\in\N$, $s\geq0$, and $k\in\Zdb$, the value of
$\Edb_s(x)$ on $Q_{s,k}$ is the arithmetic mean of
$(x(m))_{m\in Q_{s,k}}$. Hence,
\begin{equation}\label{4x2}
\Edb_s(x)\leq 2\Edb_{s+1}(x),
\qquad s\geq0,\quad x\in\N_+.
\end{equation}

We use Cuculescu's theorem for martingales \cite{C}, adapted to
reversed dyadic martingales as in \cite[Subsection 2.3]{HLX};
see also \cite[Subsection 3.1]{P} and
\cite[Subsection 1.1]{CCAP}.
Let $\widetilde{\N}\subset\N$ be the set of all
$x=(x(m))_{m\in\mathbb Z}\in
\ell^\infty_{\mathbb Z}(M\cap L^1(M))$ with finite support, that is, $\{m\in\mathbb Z:x(m)\ne0\}$ is finite. 
For $x\in\widetilde{\N}_+$ and $\lambda>0$, Cuculescu's theorem
ensures the existence of a sequence $(q_s)_{s\geq0}$ of projections in
$\N$ and an integer $s_0\geq1$, depending on $(x,\lambda)$, such that
\begin{equation}\label{4Increase}
q_0\leq q_1\leq\cdots\leq q_s\leq\cdots,
\qquad\hbox{and}\qquad q_s=1\quad\forall s\geq s_0;
\end{equation}
\begin{equation}\label{4qs}
q_s\in N_s,\qquad s\geq0;
\end{equation}
\begin{equation}\label{4Lambda-qs}
q_s\mathbb{E}_s(x)q_s\leq\lambda q_s,\qquad s\geq0;
\end{equation}
\begin{equation}\label{4L1-estimate}
\tau(q_0^\perp)\leq\frac{\|x\|_1}{\lambda}.
\end{equation}

\begin{lemma}\label{4CZ-Dec}
Let $x,\lambda$ and $(q_s)_{s\geq0}$ be as above, and set
$$
p_s:=q_{s+1}-q_s,\qquad s\geq0.
$$
Then:
\begin{itemize}
\item[(1)] The $p_s$ are pairwise orthogonal projections, and
\begin{equation}\label{4Sumps}
q_0^\perp=\sum_{s=0}^\infty p_s.
\end{equation}

\item[(2)] We have
$$
x=g+a+b,
$$
where
\[
g:=q_0xq_0,\qquad
a:=\sum_{s=0}^\infty p_sxp_s,\qquad
b:=\sum_{s=0}^\infty(p_sxq_s+q_sxp_s).
\]

\item[(3)] We have
$$
\norm{g}_\infty\leq\lambda.
$$
\end{itemize}
\end{lemma}

\begin{proof}
Part (1) follows from (\ref{4Increase}), while part (3) follows from
(\ref{4Lambda-qs}), since $\Edb_0(x)=x$.

For part (2), observe that, for every $s\geq0$,
$$
q_s=q_0+\sum_{j=0}^{s-1}p_j.
$$
Thus,
\begin{align*}
\sum_{s=0}^{\infty}p_sxq_s
&=
\sum_{s=0}^{\infty}
\left(p_sxq_0+\sum_{j=0}^{s-1}p_sxp_j\right)\\
&=
\sum_{s=0}^{\infty}p_sxq_0
+
\sum_{0\leq j<s}p_sxp_j,
\end{align*}
and similarly,
\begin{align*}
\sum_{s=0}^{\infty}q_sxp_s
&=
\sum_{s=0}^{\infty}q_0xp_s
+
\sum_{0\leq s<j}p_sxp_j.
\end{align*}
Therefore, by (\ref{4Sumps}),
\begin{align*}
a+b
&=
\sum_{s=0}^{\infty}(p_sxq_0+q_0xp_s)
+
\sum_{s,j\geq0}p_sxp_j\\
&=
(1-q_0)xq_0
+
q_0x(1-q_0)
+
(1-q_0)x(1-q_0)\\
&=x-q_0xq_0=x-g.
\end{align*}
\end{proof}

Part (2) of the previous lemma is a Calderón-Zygmund decomposition for $x$. It should be noted that we will not need the cancellation property that is often required in this type of decomposition (see, for example, \cite{CCAP, P}). Consequently, the good part $g$ is slightly different from usual.

\subsection{A non-commutative $L^2$-version of LaVictoire's lemma}\label{LaVictoire}

We begin by establishing some notation.
For any $\nu\in\ell^\infty_{\mathbb Z}$, we let
${\rm supp}(\nu)\subset\Zdb$ denote the set of integers where
$\nu$ does not vanish. For any positive $x$ in a non-commutative
$L^p$-space, we let ${\mathfrak s}(x)$ denote the support projection
of $x$, defined as the smallest $e$ such that $exe=x$.
Thus, for any
$\nu\in{\ell^\infty_{\mathbb Z}}_+$, we have
${\mathfrak s}(\nu) = \chi_{{\rm supp}(\nu)}$.

Next, for any $\nu = (\nu(n))_{n\in\mathbb Z}\in\ell^2_{\mathbb Z}$, we define
$\widetilde{\nu} = \bigl(\overline{\nu(-n)}\bigr)_{n\in\mathbb Z}$
in $\ell^2_{\mathbb Z}$. We observe that
$$
(\nu * \widetilde{\nu})(0)= \norm{\nu}_2^2\geq 0.
$$

If $\nu$ is finitely supported, we can define
$\varphi *\nu \in \ell^p_{\mathbb Z}(L^p(M))$ for any
$\varphi\in\ell^p_{\mathbb Z}(L^p(M))$
and any $1\leq p\leq \infty$.
We note that on the Hilbert space $\ell^2_{\mathbb Z}(L^2(M))$,
the adjoint of the operator $\varphi\mapsto \varphi *\nu$ is the mapping
$\varphi\mapsto \varphi * \widetilde{\nu}$.
Indeed, let $(\cdotp\,\vert\,\cdotp)$ denote the inner product on
$\ell^2_{\mathbb Z}(L^2(M))$. Then,
for any $\varphi=(\varphi(k))_{k\in\mathbb Z}$
and $\psi=(\psi(l))_{l\in\mathbb Z}$ in $\ell^2_{\mathbb Z}(L^2(M))$,
we have
\begin{align}\label{4Adjoint}
(\varphi \vert \psi *\nu) & = \sum_{k\in\mathbb Z}\tau\bigl((\psi*\nu)(k)^* \varphi(k)\bigr)\\ \nonumber
& = \sum_{k\in\mathbb Z}
\sum_{j\in\mathbb Z}\overline{\nu(j)}\,
\tau\bigl(\psi(k-j)^* \varphi(k)\bigr)\\ \nonumber
& = \sum_{k\in\mathbb Z}\sum_{l\in\mathbb Z}
\overline{\nu(k-l)}\,\tau\bigl(\psi(l)^* \varphi(k)\bigr)\\ \nonumber
& = (\varphi*\widetilde{\nu}\vert \psi). \nonumber
\end{align}

Throughout the rest of this section, we fix two non-zero
sequences $(\mu_j)_{j\geq 1}$
and $(\nu_j)_{j\geq 1}$
in $\ell^\infty_{\mathbb Z}$ such that, for all
$j\geq 1$,
\begin{equation}\label{4Support}
{\rm supp}(\mu_j)\subset \{1,\ldots, 2^j\},
\qquad
{\rm supp}(\nu_j)\subset \{1,\ldots, 2^j\},
\end{equation}
\begin{equation}\label{4Positive}
\mu_j\geq 0\qquad\hbox{and}\qquad
\mu_j-\nu_j\geq 0.
\end{equation}
For every $j\geq 1$, set
$$
r_j := {\rm Card}\bigl({\rm supp}(\mu_j)\bigr).
$$
We assume that the sequence $(r_j)_{j\geq 1}$ is non-decreasing
and that there exists a constant $K>0$ such that
\begin{equation}\label{4Domination}
\sum_{j=1}^k r_j\,\leq K r_k,\qquad k\geq 1.
\end{equation}
We further assume
that there exist two constants $K_1,K_2>0$, as well as
$\varepsilon>0$, such that
\begin{equation}\label{4Zero}
(\nu_j * \widetilde{\nu}_j)(0)\leq K_1 r_{j}^{-1},\qquad j\geq 1,
\end{equation}
and
\begin{equation}\label{4NonZero}
\bigl\vert (\nu_j * \widetilde{\nu}_j)(n)\bigr\vert\leq K_2 2^{-(1+\varepsilon)j},
\qquad j\geq 1, \, n\in\Zdb^*.
\end{equation}
The first terms of the sequence $(r_j)_{j\geq 0}$ may be zero, so
(\ref{4Zero}) should be interpreted as holding for sufficiently large $j$.

The next statement is a non-commutative analog of \cite[Lemma 3.2]{LaV}. Here, and later in the paper,
the symbol $\lesssim$ is used for inequalities
up to an absolute constant.

\begin{lemma}\label{analog LaV prop 3.1}
Let $\psi\in\widetilde{\N}$.
For every $j\geq 1$, set
$$
\sigma_j(\psi)=\sup_{k,v}\Biggl\{\biggnorm{
\sum_{m\in Q_{j,k}}v(m)\,\psi(m)}_\infty\Biggr\},
$$
where the supremum runs over all $k\in\Zdb$ and all
functions $v\colon Q_{j,k}\to\mathbb C$ such that $\norm{v}_\infty\leq 1$.
Then, for any $\varphi\in\widetilde{\N}$, we have the estimate
\begin{equation}\label{4Key}
\bigl|( \varphi*\nu_j\,\vert \,\psi*\nu_j)\bigr|
\lesssim
r_j^{-1}\,|( \varphi \,\vert\, \psi )|
+
\sigma_j(\psi) 2^{-(1+\varepsilon)j}\,
\|\varphi\|_1,\qquad j\geq 1.
\end{equation}
\end{lemma}

\begin{proof}
Using (\ref{4Adjoint}) with $\nu_j$ instead of $\nu$
and $\varphi *\nu_j$ instead of $\varphi$, we write
\[
(\varphi*\nu_j\,\vert \,\psi*\nu_j)
=
(\varphi\,\vert\,\psi*\nu_j*\widetilde{\nu}_j)
=
\sum_{n\in\mathbb Z}
\sum_{m\in\mathbb Z}
\tau\!\bigl(
\psi^*(m)(\nu_j*\widetilde{\nu}_j)(n-m)\varphi(n)
\bigr),
\]
where we used the fact that $\nu_j*\widetilde{\nu}_j$ is real. Let
\[
c_0=(\nu_j*\widetilde{\nu}_j)(0),
\]
and write
\[
\nu_j*\widetilde{\nu}_j
=
c_0\delta_0+u,
\]
where $\delta_0$ is the Dirac mass at $0$.
Then,
\[
u(0)=0,\qquad
\operatorname{supp}(u)\subset \{-2^j,\ldots, 2^j\},
\]
and by the assumptions
(\ref{4Zero}) and (\ref{4NonZero}), we have
$$
0\leq c_0\lesssim r_j^{-1}\qquad\hbox{and}\qquad
|u(n)|\lesssim 2^{-(1+\varepsilon)j},
\quad n\in\mathbb Z.
$$
Accordingly,
\begin{align*}
|\langle \varphi*\nu_j,\psi *\nu_j\rangle|
& \le c_0\left|
\sum_m\tau\bigl((\psi^*(m)\varphi(m)\bigr)
\right|
+
\left|
\sum_{n,m\in\mathbb Z}
\tau\bigl(\psi^*(m)u(n-m)\varphi(n)\bigr)
\right|\\
&\lesssim
r_j^{-1}\,|(\varphi \,\vert\, \psi )|
+
\left|
\sum_{n,m\in\mathbb Z}
\tau\bigl(\psi^*(m)u(n-m)\varphi(n)\bigr)
\right|.
\end{align*}

Let $\beta$ denote the second term on the right-hand side
of this inequality.
As in \cite{LaV}, we let
\[
\varphi_k:=\varphi\chi_{Q_{j,k}},
\qquad
\psi_l:=\psi\chi_{Q_{j,l}},
\]
for $k,l\in\Zdb$, and decompose
\(\varphi=\sum_k \varphi_k\),
\(\psi=\sum_l \psi_l\). Define
\[
Q_{j,k}^*
=
Q_{j,k-1}\cup Q_{j,k}\cup Q_{j,k+1}.
\]
This is the interval of integers with the same center as $Q_{j,k}$ and
three times its radius.

If $m\in Q_{j,k}$ and $n\notin Q_{j,k}^*$, then $|n-m|>2^j$, so that
\[
\psi^*(m)u(n-m)\varphi(n)=0.
\]
Hence,
\begin{align*}
\beta
&=
\left| \sum_k \sum_{i=-1}^1 \sum_{n,m\in\mathbb Z}
\tau\bigl(\psi_{k+i}^*(m)\,
u(n-m)\,\varphi_k(n)\bigr)\right|\\
&\le \sum_k \sum_{i=-1}^1 \left| \sum_{n,m\in\mathbb Z}
\tau\bigl(\psi_{k+i}^*(m)\,u(n-m)\,\varphi_k(n)
\bigr)\right|\\
&\leq\sum_k\sum_{i=-1}^1\|\varphi_k\|_1
\left\|\sum_{m\in\mathbb Z} u(\,\cdot-m)\,
\psi_{k+i}^*(m)\right\|_\infty\\
&\leq\sum_k\sum_{i=-1}^1\|\varphi_k\|_1\|u\|_\infty
\sigma_j(\psi)\\
&= 3 \,\sigma_j(\psi) \|\varphi\|_1 \|u\|_\infty\\
&\lesssim  \sigma_j(\psi) 2^{-(1+\varepsilon)j}\|\varphi\|_1.
\end{align*}
The estimate (\ref{4Key}) follows.
\end{proof}

\subsection{A criterion for transferring maximal inequalities}\label{Criterion}

We consider $(\mu_j)_{j\geq 1}$ and $(\nu_j)_{j\geq 1}$ as in Subsection \ref{LaVictoire}. Recall that $S$ denotes the shift operator on $\N=\ell^\infty_{\mathbb Z}(M)$.

\begin{proposition}\label{4Main}
For every $j\ge 1$, define
$$
V_j = \sum_{n} \mu_j(n) S^n 
\qquad\hbox{and}\qquad
U_j =  \sum_{n} \bigl(\mu_j(n) -\nu_j(n)\bigr)S^n.
$$
Equivalently, $V_j$ (resp. $U_j$) is the convolution operator by $\mu_j$ (resp. $(\mu_j-\nu_j)$).

Assume the following two properties:
\begin{itemize}
\item [(H1)] The sequence $(U_j)_{j\ge1}$ satisfies a weak-type
$(1,1)$-maximal inequality;
\item [(H2)] There exists a constant $C\geq 0$ such that 
$$
\bignorm{V_j\colon \ell^\infty_{\mathbb Z}(M)\longrightarrow
\ell^\infty_{\mathbb Z}(M)}\leq C,\qquad j\geq 1.
$$
\end{itemize}
Then $(V_j)_{j\ge1}$ satisfies a weak type $(1,1)$-maximal inequality.
\end{proposition}

\begin{proof}
We need to show that for any $x \in L^1(\N)=\ell^1_{\mathbb Z}(L^1(M))$,
and for any $\lambda>0$, there exists $e\in\P(\N)$ such that 
\begin{equation}\label{4Goal}
\tau(e^\perp) \lesssim \,\frac{\norm{x}_1}{\lambda}
\qquad\hbox{and}\qquad \norm{e(x*\mu_j)e}_\infty\lesssim\lambda,\quad \forall j\geq 1.
\end{equation}
Using (\ref{2Tr}), we see that if the sequence $(x*\mu_j)_{j}$ is decomposed
as a finite sum
$$
(x*\mu_j)_{j} = (\gamma_j^1)_j+\cdots +(\gamma_j^N)_j
$$
of sequences of elements of $L^1(\N)$, then 
(\ref{4Goal}) holds 
if, for every $i=1,\ldots,N$, we can find $e_i\in\P(\N)$ such that
$\tau(e_i^\perp) \lesssim \,\frac{\norm{x}_1}{\lambda}$
and $\norm{e_i\gamma_j^{i}e_i}_\infty\lesssim\lambda$ for all $j\geq 1$.
We will use this possibility of decomposing the problem without further notice.

By Remark \ref{3Pos}, it suffices to establish (\ref{4Goal})
for positive elements. By density, we may therefore restrict ourselves to $x\in\widetilde{\N}_+$.
Henceforth, we fix such an element $x$, as well as a positive real number $\lambda >0$, and apply Subsection
\ref{CZ} (in particular Lemma \ref{4CZ-Dec}) to the pair $(x,\lambda)$.

We next decompose $x$ in a way that depends on $j$. For any $j \ge 1$, define
\[
\mathcal{E}_j = \chi_{(\lambda r_j,\infty)}(x),
\]
by applying the Borel functional calculus of
$x$ to the indicator function 
$\chi_{(\lambda r_j,\infty)}$. Then, set
\[
a^{(j)}
=
\sum_{s=0}^{\infty} p_s x\mathcal{E}_j\,p_s, \qquad
b^{(j)}
=
\sum_{s=0}^{\infty}
\left(
p_s x\mathcal{E}_j q_s
+
q_s x\mathcal{E}_j p_s
\right),
\]
\[
c^{(j)} = a^{(j)} + b^{(j)},
\qquad
C^{(j)} = (a+b) - c^{(j)}.
\]
By applying the calculation in 
the proof of Lemma \ref{4CZ-Dec}, with
$x \mathcal{E}_j$ in place of $x$,
we obtain
\begin{equation}\label{4cj}
c^{(j)}  = x \mathcal{E}_j - q_0 x\mathcal{E}_j q_0.
\end{equation}
Furthermore, $C^{(j)}$ is obtained from $c^{(j)}$ by replacing $\mathcal{E}_j$ with 
\(\mathcal{E}_j^\perp = \chi_{[0,r_j]}(x)\), that is,
\begin{align}\label{4Cj}
C^{(j)} &= x\mathcal{E}_j^\perp - q_0 x \mathcal{E}_j^\perp q_0\\
& = \sum_{s=0}^{\infty} p_sx\mathcal{E}_j^{\perp}p_s
+
\sum_{s=0}^{\infty}
\left(
p_sx\mathcal{E}_j^{\perp}q_s
+
q_sx\mathcal{E}_j^{\perp}p_s
\right).\nonumber
\end{align}
Finally, we define
\begin{equation}\label{4Cj+}
C^{(j)}_{+}
=
\sum_{s=j}^{\infty} p_s x\mathcal{E}_j^{\perp}p_s
+
\sum_{s=j}^{\infty}
\left(
p_s x\mathcal{E}_j^{\perp}q_s
+
q_sx\mathcal{E}_j^{\perp}p_s
\right).
\end{equation}
and 
\begin{align*}
C^{(j)}_{-}
&
=
\sum_{s=0}^{j-1} p_sx\mathcal{E}_j^{\perp}p_s
+
\sum_{s=0}^{j-1}
\left(
p_sx\mathcal{E}_j^{\perp}q_s
+
q_sx\mathcal{E}_j^{\perp}p_s
\right).
\end{align*}
Since $x=g+a+b$, we have
\begin{equation}\label{4DoubleDec}
x
= g + c^{(j)} + C^{(j)}
\qquad\hbox{and}\qquad
C^{(j)}= C^{(j)}_{+} + C^{(j)}_{-}.
\end{equation}

Set
\[
y_j := x * \mu_j, \qquad j \ge 1.
\]
According to (\ref{4DoubleDec}), we can decompose $y_j$ as
\[
y_j
=
y_j^1+y_j^2+y_j^3+y_j^4+y_j^5,
\]
where
\[
y_j^1 = g * \mu_j,\quad
y_j^2 = c^{(j)} * \mu_j,\quad
y_j^3 = C^{(j)} * (\mu_j-\nu_j),\quad
y_j^4 = C^{(j)}_{+} * \nu_j,\quad
y_j^5 =C^{(j)}_{-} * \nu_j.
\]
It therefore suffices to prove that, for each 
$1 \le i \le 5$, there exists $e\in\P(\N)$ such that
\begin{equation}\label{ineq:maximal-condition}
\tau(e^\perp)\lesssim \,\frac{\norm{x}_1}{\lambda}
\qquad\hbox{and}\qquad
\norm{ey_{j}^{i}e}_\infty\lesssim \lambda,\quad \forall j\geq 1.
\end{equation}

We begin with the first term, $y_j^1=g*\mu_j$. 
By assumption (H2), 
\begin{align*}
\|g * \mu_j\|_\infty
&\le \|V_j\|_{\infty\to\infty}\,\|g\|_\infty \lesssim \|g\|_\infty \lesssim \lambda,
\end{align*}
where the last estimate follows from Lemma \ref{4CZ-Dec}, (3).
Hence, the desired property \eqref{ineq:maximal-condition} holds for $i=1$,
with $e=1$.

We next consider the contribution coming from $C^{(j)}$, namely the third term. 
It follows from (\ref{4Cj}) that
\[
y_j^3 = C^{(j)} * (\mu_j-\nu_j)
= x \mathcal{E}_j^{\perp} * (\mu_j-\nu_j)
- q_0 x\mathcal{E}_j^{\perp} q_0 * (\mu_j-\nu_j).
\]
Since
\[
x\mathcal{E}_j^{\perp} \leq x,
\qquad\text{and}\qquad
U_j\colon 
\N
\longrightarrow
\N
\quad \text{is positive},
\]
by the second half of (\ref{4Positive}),
we have
\[
x \mathcal{E}_j^{\perp} * (\mu_j-\nu_j)
\leq
x * (\mu_j-\nu_j).
\]
By assumption (H1), there exists a projection $e$ such that
\[
\tau_\infty(e^\perp)\lesssim\,\frac{\|x\|_1}{\lambda},\qquad
\left\|
e\bigl(x*(\mu_j-\nu_j)\bigr)e
\right\|_\infty\leq\lambda,\quad \forall\, j\ge1.\]
Consequently,
\[
\left\|e\bigl(x\mathcal{E}_j^{\perp}*(\mu_j-\nu_j)\bigr)e\right\|_\infty
\leq
\lambda,
\qquad \forall j\geq1.
\]
Similarly, there exists a projection $e'$ such that
\[
\tau_\infty(e'^{\,\perp})
\lesssim\,
\frac{\|q_0 x q_0\|_1}{\lambda},
\qquad
\left\|e'\bigl(q_0x\mathcal{E}_j^{\perp}q_0*(\mu_j-\nu_j)\bigr)e'
\right\|_\infty
\leq \lambda,
\quad \forall j\geq1.
\]
Since
\(\|q_0xq_0\|_1\leq \|x\|_1\),
it follows that \eqref{ineq:maximal-condition} holds for $i=3$,
with the projection $e\wedge e'$.

We now turn to the second term \(y_j^2 = c^{(j)} * \mu_j\). By
(\ref{4cj}), we have
\[
y_j^{2}
=x \mathcal{E}_j * \mu_j
- q_0 x\mathcal{E}_j q_0 * \mu_j.
\]
By the first half of (\ref{4Positive}), each
$q_0  x\mathcal{E}_j q_0 * \mu_j$ is positive. We therefore define a projection $e\in\P(\N)$ by
\[
e^\perp
:=
\bigvee_{j\ge 1}
{\mathfrak s}(q_0 x\mathcal{E}_j q_0 * \mu_j).
\]
Then
\(e\,
(q_0 x \mathcal{E}_j q_0 * \mu_j)\,e=0\),
for all $j \ge 1$.

Using (\ref{2Tr}) and the definition of $r_j$, we estimate $\tau_\infty(e^\perp)$ as follows:
\begin{align*}
\tau_\infty(e^\perp)
&\le
\sum_{j=1}^\infty
\tau_\infty
\bigl({\mathfrak s}(q_0 x \mathcal{E}_j q_0 * \mu_j)\bigr)\\
&\le
\sum_{j=1}^\infty
r_j\,
\tau_\infty
\bigl({\mathfrak s}(q_0 x \mathcal{E}_j q_0)\bigr)
\\
&\le
\sum_{j=1}^\infty
r_j\,
\tau_\infty
\bigl({\mathfrak s}(x\mathcal{E}_j)\bigr).
\end{align*}
Moreover, for each $j\geq 1$, we may write
\(\mathcal{E}_j=\sum_{k=j}^{\infty}(\mathcal{E}_k - \mathcal{E}_{k+1})\), and hence
 \(x \mathcal{E}_j=\sum_{k=j}^{\infty} x 
(\mathcal{E}_k - \mathcal{E}_{k+1})\). It follows that
\[
{\mathfrak s}(x\mathcal{E}_j)
\le
\bigvee_{k\ge j}
{\mathfrak s}
\!\left(x (\mathcal{E}_k - \mathcal{E}_{k+1})\right),
\]
and therefore
\[
\tau_\infty
\bigl({\mathfrak s}(x \mathcal{E}_j)\bigr)
\le
\sum_{k=j}^{\infty}
\tau_\infty\bigl(
{\mathfrak s}
\!\left( x (\mathcal{E}_k - \mathcal{E}_{k+1})\right)
\bigr).
\]
Consequently, using 
the assumption (\ref{4Domination}), we have
\begin{align*}
\tau_\infty(e^\perp)
&\le
\sum_{j=1}^{\infty}
r_j
\sum_{k=j}^{\infty}
\tau_\infty
\bigl(
{\mathfrak s}
\!\left(x (\mathcal{E}_k - \mathcal{E}_{k+1})\right)
\bigr) \\
&=
\sum_{k=1}^{\infty}
\tau_\infty
\bigl(
{\mathfrak s}
\!\left(x (\mathcal{E}_k - \mathcal{E}_{k+1})\right)
\bigr) \cdot
\sum_{j=1}^{k}r_j \\
& \lesssim
\sum_{k=1}^{\infty}
r_k\,
\tau_\infty
\bigl(
{\mathfrak s}
\!\left(x (\mathcal{E}_k - \mathcal{E}_{k+1})\right)
\bigr).
\end{align*}
Since \(x
\ge
\sum_{k=1}^{\infty}
x (\mathcal{E}_k - \mathcal{E}_{k+1})\) and
\(x (\mathcal{E}_k - 
\mathcal{E}_{k+1})\ge\lambda r_k (\mathcal{E}_k - \mathcal{E}_{k+1})\) for every $k\geq 1$, we have
\begin{align*}
\|x\|_1
&\ge
\sum_{k=1}^{\infty}
\tau_\infty
\bigl(x (\mathcal{E}_k - \mathcal{E}_{k+1})
\bigr) \ge
\sum_{k=1}^{\infty}
\lambda r_k\,
\tau_\infty (\mathcal{E}_k - \mathcal{E}_{k+1}).
\end{align*}
Finally, since \((\mathcal{E}_k - \mathcal{E}_{k+1})
\ge {\mathfrak s}
\!\left(x(\mathcal{E}_k - \mathcal{E}_{k+1})\right),\)
we conclude 
\begin{align*}
\|x\|_1\ge
\sum_{k=1}^{\infty}
\lambda r_k\,
\tau_\infty\bigl(
{\mathfrak s}
\!\left(
x(\mathcal{E}_k - \mathcal{E}_{k+1})
\right)
\bigr).
\end{align*}
Hence, \(\tau_\infty(e^\perp)
\le\frac{\|x\|_1}{\lambda}\).
This establishes \eqref{ineq:maximal-condition} for $i=2$.

To handle the fourth term,
it suffices to consider
$$
\sum_{s=j}^{\infty} q_s\, x \mathcal{E}_j^\perp\,p_s
$$
instead of $C_+^{(j)}$, since the other two sums in the definition (\ref{4Cj+}) 
are treated similarly. 
For any $s\ge j$, we note that 
$p_s\in\mathcal N_s$, by (\ref{4qs}), and that  
$p_s$ is a projection. Hence, we may write 
$$
p_s
=\sum_{k\in\mathbb Z}
e_{s,k}\,\chi_{Q_{s,k}},
$$
where $e_{s,k}\in \P(M)$ for every $k\in\Zdb$.
Then
\begin{equation}\label{4Traceps}
\tau_\infty(p_s) = 2^s\sum_{k\in\mathbb Z} \tau(e_{s,k}).
\end{equation}
Using this decomposition of $p_s$, we write
\[
q_s x\mathcal{E}_j^\perp p_s
=
\sum_{k\in\mathbb Z}
q_s x \mathcal{E}_j^\perp
e_{s,k}
\chi_{Q_{s,k}}.
\]
It follows that
\[
q_s x\mathcal{E}_j^\perp p_s*\nu_j
=
\sum_{k\in\mathbb Z}
\theta_{s,j,k}\,
e_{s,k}\,
\chi_{Q_{s,k}+\operatorname{supp}(\nu_j)},
\]
for some
\(\theta_{s,j,k}\in\N.
\) Moreover, 
\(Q_{s,k}
+\operatorname{supp}(\nu_j)
\subset
Q_{s,k}
+[1,2^j]\cap\Zdb.\)

We define
\[
e_s^\perp
=
\bigvee_{k\in\mathbb Z}
e_{s,k}\,
\chi_{Q_{s,k}+[1,2^j]\cap{\mathbb Z}},
\]
so that
\(e_s\left(q_s x\mathcal{E}_j^{\perp}p_s * \nu_j\right)e_s = 0.
\)
Moreover,
\begin{align*}
\tau_\infty(e_s^{\perp})
&\le \sum_{k\in\mathbb{Z}}
\tau_\infty\!\left(e_{s,k}\,
\chi_{Q_{s,k}+[1,2^j]\cap {\mathbb Z}}\right) \\
&=
\sum_{k\in\mathbb{Z}}
\tau(e_{s,k})\,
\operatorname{Card}\!\left(Q_{s,k}+[1,2^j]\cap\Zdb\right) \\
&=
\sum_{k\in\mathbb{Z}}
\tau(e_{s,k})\,
(2^s+2^j) \\
&\le
\sum_{k\in\mathbb{Z}}
2^{s+1}\,\tau(e_{s,k}),
\qquad \text{because } s\ge j,\\
&=
2\,\tau_\infty(p_s),\qquad \text{by (\ref{4Traceps})}.
\end{align*}
Now set \(e=\bigwedge_{s=j}^{\infty}e_s,\) so that
$$
e\left(
\sum_{s=j}^{\infty}
q_s x \mathcal{E}_j^{\perp}p_s*\nu_j
\right)e=0.
$$ 
Moreover, by the above estimate, (\ref{2Tr}) and (\ref{4Sumps}), we have
$$
\tau_\infty(e^{\perp})\le
\sum_{s=j}^{\infty}
\tau_\infty(e_s^{\perp}) 
\le
2\sum_{s=j}^{\infty}
\tau_\infty(p_s) 
\le
2\sum_{s=0}^{\infty}
\tau_\infty(p_s) \\
=
2\,\tau_\infty(q_0^\perp).
$$ 
Applying (\ref{4L1-estimate}), we obtain
$$
\tau_\infty (e^{\perp})\le\,\frac{2\|x\|_1}{\lambda}.
$$
This establishes \eqref{ineq:maximal-condition} for $i=4$.

Finally, we consider the remaining contribution from $C_-^{(j)}$, given by
\[
C_-^{(j)} = \sum_{s=0}^{j-1} p_s x\mathcal{E}_j^\perp p_s + 
\sum_{s=0}^{j-1} p_s x\mathcal{E}_j^\perp q_s + 
\sum_{s=0}^{j-1} q_s x\mathcal{E}_j^\perp p_s.
\]
As before, it suffices to consider the second sum instead of $C^{(j)}_{-}$, since the other two can be handled in the same way. We define
\[
z_j = \sum_{s=0}^{j-1} p_s x\mathcal{E}_j^\perp q_s \ast \nu_j.
\]
To show that \eqref{ineq:maximal-condition} is satisfied with
the sequence $(z_j)_{j\ge1}$ in place of $(y^5_j)_{j\geq 1}$,
it suffices to prove that
\begin{equation}\label{4zj}
\sum_{j=1}^{\infty}\|z_j\|_2^2 \lesssim \lambda\|x\|_1.
\end{equation}
Indeed, by the non-commutative Markov inequality, 
for every $j\geq 1$ there exists a projection $e_j$ such that $\tau_\infty (e_j^\perp)\leq
\lambda^{-2}\|z_j\|_2^2$ and $\|e_j z_j e_j\|_\infty \leq \lambda$. 
Letting $e = \bigwedge_j e_j$, we obtain
\[
\tau_\infty(e^\perp) \leq \sum_{j=1}^\infty 
\tau_\infty(e_j^\perp) \leq \frac{1}{\lambda^2}
\sum_{j=1}^\infty \|z_j\|_2^2.
\]
Thus, $\|e z_j e\|_\infty \leq \lambda$ for all 
$j \ge 1$, and the estimate (\ref{4zj}) yields
$\tau_\infty(e^\perp) \lesssim \lambda^{-1}\|x\|_1$, 
which is precisely the desired property in \eqref{ineq:maximal-condition}.

We now prove (\ref{4zj}). 
Expanding the sum of the squared $L^2$-norms yields
\[
\sum_{j=1}^{\infty}\|z_j\|_2^2 = \sum_{j=1}^{\infty} \sum_{0\le s_1,s_2<j} 
\left(p_{s_1} x\mathcal{E}_j^\perp q_{s_1}\ast\nu_j\,
\vert\,p_{s_2}x\mathcal{E}_j^\perp q_{s_2}\ast\nu_j \right).
\]
Recall the notation $\sigma_j(\,\cdotp)$ from Lemma \ref{4Key}.
We claim that the operators 
$p_s x \mathcal{E}_j^\perp q_s$ satisfy
\begin{equation}\label{4Claim}
\sigma_j\bigl(p_s x\mathcal{E}_j^\perp q_s\bigr)\lesssim \lambda 2^j,
\qquad 0\leq s\leq j-1.
\end{equation}
Assuming this claim for the moment, we apply 
Lemma \ref{analog LaV prop 3.1} with
$\varphi = p_{s_1} x\mathcal{E}_j^\perp q_{s_1}$ and 
$\psi= p_{s_2} x\mathcal{E}_j^\perp q_{s_2}$ to obtain
\[
\sum_{j=1}^{\infty}\|z_j\|_2^2 \lesssim \sum_{j=1}^{\infty} 
\sum_{0\le s_1,s_2<j} r_j^{-1} \left| \left( p_{s_1}
x\mathcal{E}_j^\perp q_{s_1}\,\vert\, p_{s_2}x\mathcal{E}_j^\perp q_{s_2} 
\right) \right| + 
\lambda 2^{-\varepsilon j} \|p_{s_1}x\mathcal{E}_j^\perp q_{s_1}\|_1.
\]
For $s_1\neq s_2$, the inner product vanishes because the 
projections $p_s$ are mutually orthogonal ($p_{s_2}p_{s_1}=0$). 
Thus,
\[
\sum_{j=1}^{\infty}\|z_j\|_2^2 \lesssim \sum_{j=1}^{\infty} 
\biggl( r_j^{-1} \sum_{s=0}^{j-1} \|p_s x
\mathcal{E}_j^\perp q_s\|_2^2 + 
\lambda 2^{-\varepsilon j}\,j\, \|p_s x\mathcal{E}_j^\perp q_s\|_1 \biggr).
\]
Since $\|p_s x\mathcal{E}_j^\perp q_s\|_1 \le \|x\|_1$, we 
deduce
\[
\sum_{j=1}^{\infty}\|z_j\|_2^2 \lesssim \sum_{j=1}^{\infty} 
r_j^{-1} \sum_{s=0}^{j-1} \|p_s x\mathcal{E}_j^\perp q_s\|_2^2 
+ \lambda \left( \sum_{j=1}^{\infty} 
j^2 2^{-\varepsilon j} \right) \|x\|_1.
\]
Since $\sum_{j=1}^{\infty} j^{2} 2^{-\varepsilon j}<\infty$, 
it remains to prove the estimate
\begin{equation}\label{star}
\sum_{j=1}^{\infty} r_j^{-1} \sum_{s=0}^{j-1} 
\left\|p_s x\mathcal{E}_j^\perp q_s\right\|_2^2 
\lesssim \lambda\|x\|_1.
\end{equation}

Since $\sum_s p_s\leq 1$, we have
\begin{align*}
\sum_{s=0}^{j-1} \left\|p_s x\mathcal{E}_j^\perp q_s\right\|_2^2 
& \le \sum_{s=0}^{j-1} \left\|p_sx\mathcal{E}_j^\perp\right\|_2^2\\
& =\sum_{s=0}^{j-1}\tau_\infty\bigl(\mathcal{E}_j^\perp
\varphi p_sx\mathcal{E}_j^\perp\bigr) \\
& = \tau_\infty\biggl(\mathcal{E}_j^\perp
x\biggl(\sum_{s=0}^{j-1}p_s\biggr) 
x\mathcal{E}_j^\perp\biggr)\\
& \leq \tau_\infty\bigl(\mathcal{E}_j^\perp
x^2\mathcal{E}_j^\perp\bigr) = \
\left\|x\mathcal{E}_j^\perp\right\|_2^2.
\end{align*}
In \cite[End of Section 3]{LaV},
it was observed that the assumption (\ref{4Domination}) implies a 
reversed estimate
\begin{equation}\label{4Reversed}
r_k^{-1}\gtrsim \sum_{j=k}^\infty r_{j}^{-1},\qquad  k\geq 1.
\end{equation}
Setting $r_0=0$, we can express the $L^1$-norm of $x$ and 
complete the estimate as follows. Since 
$1\geq \sum_{k=1}^\infty (\E_{k-1} -\E_k)$, we have 
$$
x\geq \sum_{k=1}^{\infty}\frac{r_{k}^{-1}}{\lambda}
x^2(\E_{k-1} -\E_k),
$$
which implies
$$
\norm{x}_1\geq \sum_{k=1}^{\infty}\frac{r_{k}^{-1}}{\lambda}
\tau_\infty\bigl(x^2(\E_{k-1} -\E_k)\bigr).
$$
Consequently, using (\ref{4Reversed}),
\begin{align*}
\| x\|_1 
&\gtrsim\frac{1}{\lambda} \sum_{k=1}^{\infty} \sum_{j=k}^{\infty} r_j^{-1} \tau_\infty\bigl(x^2(\E_{k-1} -\E_k)\bigr) \\
& =\frac{1}{\lambda} \sum_{j=1}^{\infty} r_j^{-1} \sum_{k=1}^{j} \tau_\infty\bigl(x^2(\E_{k-1} -\E_k)\bigr)\\
& = \frac{1}{\lambda} \sum_{j=1}^{\infty} r_j^{-1} 
\| x\mathcal{E}_j^\perp\|_2^2.
\end{align*}
This establishes \eqref{star} and completes the proof of the primary estimate
(\ref{4zj}).

It remains to prove the claim (\ref{4Claim}).
Set $\psi = p_s x\mathcal{E}_j^\perp q_s$ for $0\le s\le j-1$. 
Let $k\in\mathbb Z$ and consider a function $v:Q_{j,k}\longrightarrow \mathbb C$. 
By decomposing $v$ into four parts, we may assume without 
loss of generality that $v\ge 0$. Partition $Q_{j,k}$ as
\[
Q_{j,k} = \bigsqcup_{k'\in\Lambda} Q_{s,k'}, \qquad \text{with } \operatorname{Card}(\Lambda)=2^{\,j-s}.
\]
Accordingly,
\[
\sum_{m\in Q_{j,k}} v(m)\psi(m) = 
\sum_{k'\in\Lambda} \sum_{m\in Q_{s,k'}} v(m)\psi(m),
\]
and hence
\[
\left\| \sum_{m\in Q_{j,k}} v(m)\psi(m) \right\|_\infty 
\le \sum_{k'\in\Lambda} \left\| \sum_{m\in Q_{s,k'}} v(m)
\psi(m) \right\|_\infty.
\]

Since $p_s, q_s \in \mathcal N_s$, by (\ref{4qs}),
they are constant on each $Q_{s,k'}$. Let $p_{s,k'} \in M$ and 
$q_{s,k'} \in M$ denote their respective constant values 
on this set. We have $q_{s+1}q_s=q_s$, by (\ref{4Increase}),
and hence $p_sq_{s+1}=p_s$.
Using these two relations, we obtain
\[
\sum_{m\in Q_{s,k'}} v(m)\psi(m)
=p_{s,k'}\biggl(\sum_{m\in Q_{s,k'}}
\bigl(q_{s+1} x\mathcal{E}_j^\perp v q_{s+1}\bigr)(m)
\biggr) q_{s,k'}.
\]
Taking $L^\infty$-norms, we deduce
\[
\left\| \sum_{m\in Q_{s,k'}} v(m)\psi(m) \right\|_\infty
\le
\left\|
\sum_{m\in Q_{s,k'}}
\bigl(q_{s+1}x\mathcal{E}_j^\perp v q_{s+1}\bigr)(m)
\right\|_\infty.
\]
Since $v\geq 0$, we have 
$0\le x\mathcal{E}_j^\perp v \le x\|v\|_\infty$. 
It follows that
\[
0 \le q_{s+1}x\mathcal{E}_j^\perp v q_{s+1}
\le
\|v\|_\infty\,q_{s+1} x q_{s+1}.
\]
Therefore,
\[
\left\| \sum_{m\in Q_{s,k'}} v(m)\psi(m) \right\|_\infty
\le
\|v\|_\infty
\left\|
\sum_{m\in Q_{s,k'}}
\bigl(q_{s+1} x q_{s+1}\bigr)(m)
\right\|_\infty.
\]
Since the arithmetic mean
$$
\frac{1}{2^s}\sum_{m\in Q_{s,k'}}
\bigl(q_{s+1} x q_{s+1}\bigr)(m)
$$
is precisely the value of
$\mathbb{E}_s(q_{s+1}x q_{s+1})$ on $Q_{s,k'}$, 
it follows that
\[
\left\| \sum_{m\in Q_{s,k'}} v(m)\psi(m) \right\|_\infty
\le
\|v\|_\infty
\bignorm{\mathbb{E}_s(q_{s+1}x q_{s+1})}_\infty
\,2^s.
\]
Since $q_{s+1}\in\mathcal N_s$, by (\ref{4qs}), we have
$\mathbb{E}_s(q_{s+1}x q_{s+1})
=q_{s+1}\mathbb{E}_s(x)q_{s+1}$. Applying (\ref{4x2}), we deduce 
$\mathbb{E}_s(q_{s+1}x q_{s+1})
\le
2\,q_{s+1}\mathbb{E}_{s+1}(x)q_{s+1}$,
and hence
\[
\bignorm{\mathbb{E}_s(q_{s+1}x q_{s+1})}_\infty
\le
2\bignorm{q_{s+1}\mathbb{E}_{s+1}(x)q_{s+1}}_\infty.
\]
By (\ref{4Lambda-qs}), \(q_{s+1}\mathbb{E}_{s+1}(x)q_{s+1}
\le
\lambda q_{s+1}\), and therefore 
\(\bignorm{q_{s+1}\mathbb{E}_{s+1}(x)q_{s+1}}_\infty
\leq
\lambda\). Consequently,
\[
\left\| \sum_{m\in Q_{s,k'}} v(m)
\psi(m) \right\|_\infty
\le
2^{s+1}\lambda\,\|v\|_\infty.
\]
Summing over the elements of $\Lambda$, we conclude
\[
\left\| \sum_{m\in Q_{j,k}} v(m)\psi(m) \right\|_\infty
\le
\operatorname{Card}(\Lambda)\,2^{s+1}\lambda\|v\|_\infty
\le
2^{j+1}\lambda\|v\|_\infty.
\]
This proves (\ref{4Claim}) and completes the proof of the proposition.
\end{proof}
\section{An individual random ergodic theorem}\label{5}

We now prove the $L^1$ endpoint result for the random ergodic averages. The
$L^2(M)$ convergence is already provided by Theorem \ref{2Back}; hence,
by the non-commutative Banach principle, Corollary \ref{3Y2}, it remains to establish the
corresponding weak-type $(1,1)$-maximal inequality. The latter will be
obtained in two steps. We first control the deterministic weighted
averages by Yeadon's maximal inequality and summation by parts. We then
transfer the resulting estimate to the random averages by separating
each random kernel into its deterministic mean and its centred part
and applying the criterion of Proposition \ref{4Main}.

\begin{lemma}\label{5Av-Y}
Let $T\colon M\to M$ be a positive Dunford-Schwartz operator, let $\alpha\in(0,1)$ and consider $W_N$ defined by (\ref{2WN}) for all $N\geq 1$. Then the sequence
$$
\biggl(\frac{1}{W_N}\,\sum_{n=1}^N
n^{-\alpha} T_1^n\biggr)_{N\geq 1}
$$
satisfies a weak-type
$(1,1)$-maximal inequality.
\end{lemma}

\begin{proof}
We compare the weighted averages with the usual Ces\`aro averages
$A_N(T)$ from \eqref{2Cesaro}. By the main result of \cite{Ye}, the
sequence $(A_N(T_1))_{N\geq1}$ satisfies a weak-type $(1,1)$-maximal
inequality. Thus there is a constant $C\geq0$ such that, for every
$x\in L^1(M)$ and every $\lambda>0$, there exists a projection
$e\in\P(M)$ satisfying $\tau(e^\perp)\leq C\lambda^{-1}\norm{x}_1$
and 
\begin{equation}\label{5Yeadon}
\bignorm{e[A_N(T)](x)e}_\infty\leq\lambda,\qquad N\geq 1.
\end{equation}

Set
\[
\lambda_m := m \bigl(m^{-\alpha} - (m+1)^{-\alpha}\bigr),
\qquad m\geq 1.
\]
The summation-by-parts calculation from \cite[Lemma 4.2]{LZ} gives
$$
\frac{1}{W_N} \sum_{n=1}^N n^{-\alpha} T_1^n(x) =\,\frac{N^{1-\alpha}}{W_N} 
[A_N(T_1)](x) + 
\left(\frac{\sum_{m=1}^{N-1} \lambda_m}{W_N}\right)
\left(\sum_{m=1}^{N-1} \lambda_m\right)^{-1}
\sum_{m=1}^{N-1} \lambda_m [A_m(T_1)](x).
$$
Multiplying on both sides by $e$, taking the $L^\infty$-norm and using the triangle inequality we obtain
\begin{align*}
\Bignorm{e\Bigl(&\frac{1}{W_N}  \sum_{n=1}^N n^{-\alpha} T_1^n(x)\Bigr)e}_\infty \\
&\leq \,\frac{N^{1-\alpha}}{W_N} \bignorm{e[A_N(T_1)](x)e}_\infty
+ \left(\frac{\sum_{m=1}^{N-1} \lambda_m}{W_N}\right)
\left(\sum_{m=1}^{N-1} \lambda_m\right)^{-1}
\sum_{n=1}^N \lambda_m \bignorm{e[A_m(T_1)](x)e}_\infty.
\end{align*}
Hence, by (\ref{5Yeadon}),
$$
\Bignorm{e\Bigl(\frac{1}{W_N}  \sum_{n=1}^N n^{-\alpha} T_1^n(x)\Bigr)e}_\infty
\leq \lambda\Biggl( \frac{N^{1-\alpha}}{W_N}
+ \left(\frac{\sum_{m=1}^{N-1} \lambda_m}{W_N}\right)\Biggr).
$$
Since 
$$
 \frac{N^{1-\alpha}}{W_N}
+ \left(\frac{\sum_{m=1}^{N-1} \lambda_m}{W_N}\right)\ \longrightarrow 1\quad
\hbox{as } \ N\to\infty,
$$
by \eqref{2WNEquiv} (see also \cite{LZ}), the preceding estimate yields the  estimate
$$
\Bignorm{e\Bigl(\frac{1}{W_N} \sum_{n=1}^N n^{-\alpha} T_1^n(x)\Bigr)e}_\infty
\lesssim \lambda.
$$
This shows the desired weak-type
$(1,1)$-maximal inequality.
\end{proof}

Lemma \ref{5Av-Y} provides the weak-type maximal inequality for the deterministic weighted averages, which is used in proving our main result.

\begin{theorem}\label{5Main}
Let $(M,\tau)$ be a tracial von Neumann algebra.
Let $J\colon M\to M$ be a trace-preserving Jordan 
isomorphism. Assume that $\alpha\in\bigl(0,\frac12\bigr)$, where $\alpha$ is the parameter from (\ref{2Alpha}). Then there exist a contractive projection
$P_1\colon L^1(M)\to L^1(M)$ and a measurable set
$\Omega'\subset\Omega$ with $\Pdb(\Omega')=1$ such that, for every
$\omega\in\Omega'$ and every $x\in L^1(M)$, we have
\begin{itemize}
\item[(1)]
$$
\frac{1}{W_N}\sum_{n=1}^N X_n(\omega)\,J_1^n(x)
\,\longrightarrow\,P_1(x)\quad\hbox{b.a.u.}
\quad\hbox{as }N\to\infty.
$$

\item[(2)]
$$
\frac{1}{m}\sum_{k=1}^m J_1^{n_k(\omega)}(x)
\,\longrightarrow\,P_1(x)\quad\hbox{b.a.u.}
\quad\hbox{as }m\to\infty.
$$
\end{itemize}
\end{theorem}

\begin{proof}
We fix a trace-preserving Jordan isomorphism
$J\colon M\to M$ and first prove part (1).
Let $P_2\colon L^2(M)\to L^2(M)$ be the operator appearing
in Theorem \ref{2Back} applied to $T=J$ and $p=2$.
Let $P_1\colon L^1(M)\to L^1(M)$ be its $L^1$-extension,
as in Lemma \ref{2Gamma1}. In particular, $P_1$ is a
contractive projection.
 
Let $\Omega_0\subset\Omega$ be 
the set of all $\omega\in\Omega$ such that 
(\ref{2SLLN}) holds and  such that 
$$
\frac{1}{W_N}\sum_{n=1}^N X_n(\omega)\, J_2^n(x)
\,\longrightarrow P_2(x)\quad\hbox{b.a.u.}
\quad\hbox{as } N\to\infty,
$$
for all $x\in L^2(M)$.  Since (\ref{2SLLN}) holds almost surely,
Theorem \ref{2Back} gives $\Pdb(\Omega_0)=1$.
By Corollary \ref{3Y2}, (\ref{2SLLN}) and Remark \ref{2tn},
it suffices to show that, for almost every $\omega\in\Omega_0$,
the sequence
\begin{equation}\label{5WL1}
\biggl(\frac{1}{W_N}\sum_{n=1}^N X_n(\omega)\,J_1^n\biggr)_{N\geq 1}
\quad\hbox{satisfies a weak-type }(1,1)\hbox{-maximal inequality}.
\end{equation} 
To this end, for $\omega\in\Omega_0$, define
$$
V_{j,\omega} = \frac{1}{W_{2^j}}\,
\sum_{n=1}^{2^j} X_n(\omega) S^n
\qquad\hbox{and}\qquad
V'_{j,\omega} = \frac{1}{W_{2^j}}\,
\sum_{n=1}^{2^j} X_n(\omega) J^n,
$$
which we extend to $\ell^1_{\mathbb Z}(L^1(M))$ and $L^1(M)$,
respectively.

For any $N\geq 2$, let 
$j\geq 1$ be the unique integer such that $2^{j-1}<N\leq 2^j$. Since $J$ is positive, we have
$$
0\leq \sum_{n=1}^N X_n(\omega)\, J^n\,\leq 
\sum_{n=1}^{2^j} X_n(\omega)\, J^n,
$$
hence
$$
0\leq\,\frac{1}{W_N}\,
\sum_{n=1}^N X_n(\omega)\, J^n
\leq \biggl(\frac{W^{2^j}}{W^{2^{j-1}}}\biggr)
V'_{j,\omega}.
$$
Moreover, $(W_{2^{j-1}})^{-1}W_{2^j}\to 2^{1-\alpha}$ as $j\to\infty$, by (\ref{2WNEquiv}). 
Thus, if $(V'_{j,\omega})_{j\geq1}$ satisfies a weak-type
$(1,1)$-maximal inequality, then
$\bigl(W_N^{-1}\sum_{n=1}^N X_n(\omega)J^n\bigr)_{N\geq1}$
satisfies a weak-type $(1,1)$-maximal inequality for positive
elements. By Remark \ref{3Pos}, this implies (\ref{5WL1}). By Lemma \ref{lem:transference}, it therefore suffices to show that,
for almost every $\omega\in\Omega_0$,
\begin{equation}\label{5WL1bis}
\bigl(V_{j,\omega}\bigr)_{j\geq1}
\quad\hbox{satisfies a weak-type }(1,1)\hbox{-maximal inequality}.
\end{equation}

We now prove (\ref{5WL1bis}) using Proposition \ref{4Main}. For $\omega\in\Omega_0$ and $j\geq1$, define
$\mu_{j,\omega}$ and $\nu_{j,\omega}$ in
$\ell^\infty_{\mathbb Z}$ by
\begin{align*}
\mu_{j,\omega}(n) & = \frac{1}{W_{2^j}}X_n(\omega),\qquad 1\leq n\leq 2^j,\\
\nu_{j,\omega}(n) & = \frac{1}{W_{2^j}}\bigl(X_n(\omega)-n^{-\alpha}
\bigr),\qquad 1\leq n\leq 2^j,\\
\mu_{j,\omega}(n) & =\nu_{j,\omega}(n)=0,\qquad n\notin[1,2^j].
\end{align*}
Then (\ref{4Support}) and (\ref{4Positive}) 
are satisfied when $\mu_j=\mu_{j,\omega}$
and $\nu_j=\nu_{j,\omega}$.

Set $r_{j,\omega} = {\rm Card}\bigl({\rm supp}(\mu_{j,\omega})\bigr)$.
It follows from \cite[Section 4]{LaV} that there exists
$\Omega_1\subset\Omega_0$ with $\Pdb(\Omega_1)=1$ such that,
for every $\omega\in\Omega_1$, the conditions
(\ref{4Domination}), (\ref{4Zero}), and (\ref{4NonZero}) hold with
$r_j=r_{j,\omega}$ and $\nu_j=\nu_{j,\omega}$. This is the point in the proof where the hypothesis $\alpha<\frac12$ is used.

Consider 
$$
U_{j} : = \sum_n\bigl(\mu_{j,\omega}(n)-\nu_{j,\omega}(n)\bigr) S^n\,
=\,\frac{1}{W_{2^j}}\sum_{n=1}^{2^j} n^{-\alpha} S^n.
$$
By Lemma \ref{5Av-Y}, the sequence $(U_j)_{j\geq1}$ satisfies
assumption (H1) of Proposition \ref{4Main}. On the other hand,
\cite[Lemma 5.5]{LZ} implies that, for almost every
$\omega\in\Omega_1$,
$$
\sup_j
\bignorm{V_{j,\omega}\colon
\ell^\infty_{\mathbb Z}(M)\longrightarrow\ell^\infty_{\mathbb Z}(M)}
<\infty.
$$
Thus, $(V_{j,\omega})_{j\geq1}$ satisfies assumption (H2) of
Proposition \ref{4Main}. Consequently, (\ref{5WL1bis}) holds for
almost every $\omega\in\Omega_0$, which proves part (1).

Finally, the argument in the last part of the proof of \cite[Theorem 5.2]{LZ} shows that part (2) follows from part (1).
\end{proof}

\bigskip\noindent
{\bf Acknowledgement.} Part of this work was developed during the HAVNA workshop at ICMAT, where the authors worked together. The authors gratefully acknowledge the support of the Heilbronn Institute for Mathematical Research and the UKRI/EPSRC Additional Funding Programme. The first and third authors gratefully acknowledge the support of EIPHI Graduate school (contract ANR-17-EURE-0002). The third author warmly thanks Vlad Gheorghiu for his kind hospitality during her visit to the Institute for Quantum Computing in Waterloo, including the wonderful keyboard that made writing this paper all the more joyful.

\bigskip   

\end{document}